\documentclass[11pt]{article}
\usepackage{amsmath,amsfonts,amsthm,amssymb, mathtools}
\usepackage{mathrsfs, graphicx,color,latexsym, tikz, calc,
}
\usetikzlibrary{shadows}
\usetikzlibrary{patterns,arrows,decorations.pathreplacing}
\newtheorem{thm}{Theorem}[section]
\newtheorem{lem}[thm]{Lemma}

\newtheorem{re}[thm]{Remark}

\newtheorem{prop}[thm]{Proposition}
\newtheorem{exa}[thm]{Example}

\newcommand{\field}[1]{\mathbb{#1}}
\newcommand{\R}{\field{R}}

\newcommand{\N}{\field{N}}
\newcommand\scalemath[2]{\scalebox{#1}{\mbox{\ensuremath{\displaystyle #2}}}}
\newcommand{\diam}{\mathrm{diam}}

\title{\bf \Large A Hoffman's Theorem: a revisit with new discovery}

\author{
{\small  Jianfeng Wang$^{a,}$\footnote{Corresponding author.
\newline{\it \hspace*{5mm}Email addresses:} jfwang@sdut.edu.cn (J. Wang), jwang66@aliyun.com (J. Wang), maurizio.brunetti@unina.it (M. Brunetti).}\;,\;\ \ Jing Wang$^b$, \ \ Maurizio Brunetti$^c$}\\[2mm]
\footnotesize $^a$School of Mathematics and Statistics, Shandong University of Technology, Zibo 255049, China\\
\footnotesize $^b$Department of Applied Mathematics, Northwestern Polytechnical University, Xi'an, 710072, China\\
\footnotesize $^c$Department of Mathematics and Applications, University of Naples `Federico II', Italy}

\date{ }

\begin{document}
\maketitle
\begin{abstract}
In 1972, A. J. Hoffman proved his celebrated theorem concerning  the limit points of spectral radii of non-negative symmetric integral matrices less than $\sqrt{2+\sqrt{5}}$. In this paper, after giving a new version of Hoffman's theorem, we get  two generalized versions of it applicable to non-negative symmetric matrices with fractional elements. As a corollary, we obtain another alternative version about the limit points of spectral radii of (signless) Laplacian matrices of graphs less than $2+ {\tiny \frac{{\;}1{\;}}{3}\left((54 - 6\sqrt{33})^{\frac{{\;}1{\;}}{3}} + (54 + 6\sqrt{33})^{\frac{{\;} 1{\;}}{3}} \right)}$. We also discuss how our approach could be fruitfully employed  to investigate equiangular lines.\\

\end{abstract}

\baselineskip=0.2in

\section{Introduction}
 A. J. Hoffman determined in  \cite{hoffman}  the limit points of spectral radii of non-negative symmetric integral matrices less than
$\sqrt{2+\sqrt{5}}$. His strategy essentially consisted in reducing the problem to the adjacency matrices of graphs. Hoffman's result, which is Theorem~\ref{hoffman} in this paper, pioneered a fruitful investigation on limit points of graphs eigenvalues. For instance, Shearer \cite{she} proved that all numbers greater than $2+\sqrt{5}$ are limit points of spectral radii of adjacency matrices, and many results in the same vein were obtained later on concerning the eigenvalues of various (di)graph matrices, including the adjacency and the (signless) Laplacian matrix (see, for instance, \cite{doob-lim,hof-least,zhang-chen}).

Meanwhile, the graphs with adjacency spectral radius at most $\sqrt{2+\sqrt{5}}$
has been gradually characterized. Smith \cite{smith} determined all graphs whose adjacency spectral radius doest not exceed $2$. Brouwer and Neumaier \cite{BN} identified  all graphs whose spectral radii of adjacency matrices are between 2 and $\sqrt{2+\sqrt{5}}$ completing a research project started by Cvetkovi\'c, Doob and Gutman \cite{CDG}. The above investigation has been known as {\em the Hoffman program} of graphs with respect to the adjacency matrix. For more details on the Hoffman program of graphs, see \cite[Section~1.3.3]{S}, \cite[Section~3.3]{BCKW} and a new survey \cite{JFW-JW-MB}.

Surprisingly, Hoffman's theorem and the results stated above turned out useful to study equiangular lines  in the $n$-dimensional Euclidean space, i.e\ a family of lines through the origin such that the angle between any pair of them is the same \cite{jiang-poly}.

In this paper we provide a new version of Hoffman's theorem and two generalized results of it. Hoffman's result concerns non-negative symmetric integral matrices. Our generalizations, stated in Section~2, apply to certain non-negative symmetric  matrices with fractional elements. Hoffman proved his theorem by showing that he could restrict to the set of the matrices involved to the $(0,1)$ symmetric matrices with null main diagonal, i.e.\ to the
adjacency matrices of graphs. The strategy proposed here is to use the convex linear combination of adjacency matrix and degree matrix of graphs and the software {\em Mathematica\textsuperscript{\textregistered}} plays a pivotal role in some proofs.

Let $G= (V(G),E(G))$ be a simple and undirected graph with vertex set $V(G) = \{v_1, \dots, v_n\}$ and edge set $E(G)$.  The well-known {\it adjacency matrix}, denoted by $A(G) = (a_{ij})_{n
\times n}$, is the $(0,1)$-symmetric matrix with $a_{ij}=1$ if $v_iv_j \in E(G)$ and $a_{ij}=0$ otherwise. For the vertex $v \in V(G)$, $d(v)$
denotes the degree of a vertex $v \in V(G)$, and $D(G)={\rm diag}(d(v_1),d(v_2),\cdots,d(v_n))$ is the degree matrix of $G$. The {\it Laplacian}
and the {\it signless Laplacian} matrices are respectively defined as $L(G) = D(G) - A(G)$ and $Q(G) = D(G) + A(G)$. Generally, the Laplacian matrix stems from the
celebrated Kirchhoff's matrix tree theorem  which is very recently generalized to directed and weighted graphs by Leenheer \cite{lee}. Surely,
these three ones stand among the most widely studied graph matrices.

\medskip
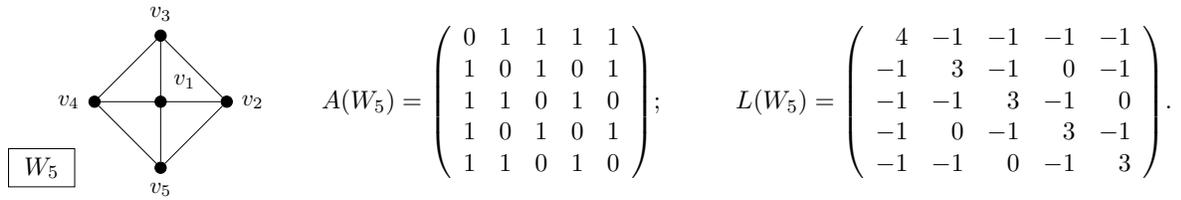
\begin{figure}[h!]
\begin{center}
 \resizebox{1\textwidth}{!}
{\begin{tikzpicture}[vertex1_style/.style={circle,draw,minimum size=0.17 cm,inner sep=0pt, fill=black},vertex2_style/.style={circle,draw,minimum size=0.2 cm,inner sep=0pt}, nonterminal/.style={
rectangle,
minimum size=2mm,
thin,
draw=black,
top color=white, 
bottom color=white!50!white!50, 
font=\itshape
}]


    \node[vertex1_style, label=above right:\small$v_1$] (a0) at  (1,1) {};
     \node[vertex1_style, label=right:\small$v_2$] (a1) at  (2,1) {};
      \node[vertex1_style, label=above:\small$v_3$] (a2) at  (1,2) {};
       \node[vertex1_style, label=left:\small$v_4$] (a3) at  (0,1) {};
              \node[vertex1_style, label=below:\small$v_5$] (a4) at  (1,0) {};
           \draw (a1)--(a3);
            \draw (a2)--(a4);
             \draw (a2)--(a3);
             \draw (a4)--(a3);
               \draw (a1)--(a2);
              \draw (a1)--(a4);
\node[nonterminal]   (Z05) at (-.8,0) {$\; W_5 \;$};
\node      (a)         at (6,1) {$A(W_5)=\left(
  \begin{array}{ccccc}
    0 & 1 & 1 & 1 & 1 \\
    1 & 0 & 1 & 0 & 1 \\
    1 & 1 & 0 & 1 & 0 \\
    1 & 0 & 1 & 0 & 1 \\
    1 & 1 & 0 & 1 & 0 \\
  \end{array}
  \right)$;};
  \node      (b)         at (13,1) {$L(W_5)=\left(
  \begin{array}{ccccc}
    \phantom{-}4 & -1 & -1 & -1 & -1 \\
    -1 & \phantom{-}3 & -1 & \phantom{-}0 & -1 \\
    -1 & -1 & \phantom{-}3 & -1 & \phantom{-}0 \\
    -1 & \phantom{-}0 & -1 & \phantom{-}3 & -1 \\
    -1 & -1 & \phantom{-}0 & -1 &\phantom{-}3 \\
  \end{array}
  \right)$.};
\end{tikzpicture} }
\end{center}
\vspace{-6mm}
  \caption{ \label{fig3}  \small  Adjacency and Laplacian matrix of the wheel graph $
 W_5$,  }
  \end{figure}

Let $M = M(G)$ any matrix associated with $G$. The {\it $M$-polynomial} of $G$ is defined as {\rm det}$(\lambda{I} - M)$, where $I$ is the identity matrix. The {\it
$M$-spectrum} of $G$ is a multiset consisting of the eigenvalues of its graph matrix $M$, and the largest absolute value of them is called the {\it
$M$-spectral radius of $G$}. We denote it by $\rho\!_{_M}\!(G)$.

In order to state the celebrated theorem by Hoffman, we recall that a real number $\gamma(M)$ is said to be a {\it limit point} of the  $M$-spectral radius of graphs -- or simply an {\it $M$-limit point} -- if there exists a sequence of graphs $\{G_k\, |\, k\in \mathbb{N}\}$ such that  $$\rho\!_{_M}\!(G_i) \neq \rho\!_{_M}\!(G_j) \quad \text{whenever  $i \neq j$}, \quad \text{and}  \quad \lim_{k \rightarrow \infty}\rho\!_{_M}\!(G_k)=\gamma(M).$$

\begin{thm}[Hoffman's theorem]\label{hoffman} Let $\tau$ denote the number $(\sqrt{5}+1)/2$. For $n \in \N$, let $\eta_n=\beta_n^{\frac{1}{2}}+\beta_n^{-\frac{1}{2}}$, where
$\beta_n$ is the positive root of
\begin{equation}\label{puah0}
\phi_n(x)=x^{n+1}-(1+x+x^2+\cdots+x^{n-1}).
\end{equation} The numbers $2=\eta_1<\eta_2<\cdots$ are precisely the limit points of the $A$-spectral radius of graphs smaller than
$$\lim_{n\rightarrow \infty}\eta_n = \tau^{\frac{1}{2}}+\tau^{-\frac{1}{2}}= \sqrt{2+\sqrt{5}}.$$
\end{thm}

The new and generalized versions of Theorem~\ref{hoffman} involve the {\it $A_{\alpha}$-matrix}  of a graph $G$ (see \cite{niki}), i.e\  a convex linear combination $A_{\alpha}(G)=\alpha D(G)+(1-\alpha)A(G)$, where $\alpha \in [0,1]$. Clearly, $A(G) = A_0(G), Q(G)=2A_{1/2}(G)$ and $L(G)=\frac{1}{\alpha-\beta}(A_\alpha(G)-A_\beta(G))$ for all $\alpha \not=\beta$. The $A_{\alpha}$-matrix is non-negative symmetric with fractional elements if $\alpha \in (0,1)$. For $\alpha \not=1$,  the spectral theory of $A_{\alpha}$ is equivalent to the one arising from the {\it general matrix} $M_{\gamma}$, defined in \cite{LHGL} as
$ M_{\gamma} (G) = \gamma D(G) + A(G)$ for every $\gamma \geq 0$, the matrices $A_{\alpha}(G)$ and $M_{\gamma} (G)$ being proportional. 
 Note that the $A_{\alpha}$-matrix of a connected graph is non-negative irreducible. By the Frobenius-Perron Theorem \cite[Theorem 8.4.4, pp. 508]{horn-john}, the $A_\alpha$-spectral radius is an algebraically simple eigenvalue of $A_{\alpha}$-matrix associated with a positive eigenvector.\medskip

\begin{exa} For the graph $W_5$ in Figure 1, set $\alpha=\frac{\;1\;}{3}$ and $\alpha=\frac{\;3\;}{4}$. The corresponding $A_\alpha$-matrices of $G$ take the following form.
$$
A_{\frac{1}{3}}(G)=\left(
  \begin{array}{ccccc}
    \frac{4}{3} & \frac{2}{3} & \frac{2}{3} & \frac{2}{3} & \frac{2}{3} \\[1.1mm]
    \frac{2}{3} & 1           & \frac{2}{3} & 0           & \frac{2}{3} \\[1.1mm]
    \frac{2}{3} & \frac{2}{3} & 1           & \frac{2}{3} & 0 \\[1.1mm]
    \frac{2}{3} & 0           & \frac{2}{3} & 1           & \frac{2}{3} \\[1.1mm]
    \frac{2}{3} & \frac{2}{3} & 0           & \frac{2}{3} & 1 \\[1.1mm]
  \end{array}
  \right)
\quad \mbox{and} \quad
A_{\frac{3}{4}}(G)=\left(
  \begin{array}{ccccc}
    3           & \frac{1}{4} & \frac{1}{4} & \frac{1}{4} & \frac{1}{4} \\[1.1mm]
    \frac{1}{4} & \frac{9}{4} & \frac{1}{4} & 0           & \frac{1}{4} \\[1.1mm]
    \frac{1}{4} & \frac{1}{4} & \frac{9}{4} & \frac{1}{4} & 0  \\[1.1mm]
    \frac{1}{4} & 0           & \frac{1}{4} & \frac{9}{4} & \frac{1}{4} \\[1.1mm]
    \frac{1}{4} & \frac{1}{4} & 0           & \frac{1}{4} & \frac{9}{4} \\[1.1mm]
  \end{array}
  \right).
$$
It turns out that $\rho_{A_{\frac{1}{3}}}(G) = (11+\sqrt{73})/6$, and $\rho_{A_{\frac{3}{4}}}(G) =(23+\sqrt{17})/8$.
\end{exa}

A first attempt to study limit points of the $A_{\alpha}$-spectral radius of graphs has been performed in \cite{WWXB}, where it has been shown that the smallest limit point for the $A_{\alpha}$-spectral radius of graphs is $2$. Results in this paper are a natural second step. In fact we identify all the possibile smallest $A_{\alpha}$-limit points which are bigger than $2$. The $A_{\alpha}$-matrix really merges the adjacency and the signless Laplacian spectral properties. Infact, we shall be able to deduce both Theorem~\ref{hoffman} and Theorem~\ref{ult} from a single statement.

The remainder of the paper is structured as follows. Section~2 contains the statements of our main results. Section~3 contains some technical results and preliminaries needed for the proofs of  Theorems \ref{Aa-main2} and \ref{Aa-main1}, presented in Section~4. Section 5 contains a re-formulation of the known results on the limits points on  (signless) Laplacian spectral radius of graphs.  In the concluding Section 6 we  propose two problems to be addressed in future studies and discuss the potential applications to the problem of estimating the maximum number of equiangular lines in an $n$-dimensional Euclidean space.

\section{New discoveries}

Let $P_2(P_m,P_n)$ be the graph depicted in Fig.~\ref{fig0}, where $P_n$ denotes the {\it path} of order $n$. From \cite[Proposition~3.6]{hoffman} and Proposition~\ref{alpha-p2pnpn} in Section 3, it follows that
\begin{equation}\label{psi}
\Psi(\alpha) = \lim\limits_{n\rightarrow\infty} \rho_{_{A_{\alpha}}} (P_2(P_n,P_n))
\end{equation}
is a well-defined strictly increasing continuous function $\Psi : [0,1] \longrightarrow \R$, such that
$$ \Psi(0) = \sqrt{2+\sqrt{5}} \qquad \text{and} \qquad  \Psi(1) = 3.$$
A closed-form expression for $\Psi(\alpha)$ is determined in \eqref{PSI}.

\medskip
\begin{figure}[h]
\begin{center}
 \resizebox{0.35\textwidth}{!}
{\begin{tikzpicture}[vertex1_style/.style={circle,draw,minimum size=0.17 cm,inner sep=0pt, fill=black},vertex2_style/.style={circle,draw,minimum size=0.2 cm,inner sep=0pt}]

\draw [decorate,decoration={brace,amplitude=5pt,mirror,raise=2ex}]
  (3.9,0) -- (7.6,0) node[midway,yshift=-2em]{\scriptsize$n$};
  \draw [decorate,decoration={brace,amplitude=5pt,raise=-1ex}]
  (3.9,1.5) -- (7.6,1.5) node[midway,yshift=.8em]{\scriptsize$m$};
 \node[vertex1_style, ] (a1) at  (4,1) {};
  \node[vertex1_style,] (a2) at  (5,1) {};
    \node[vertex1_style, ] (a3) at  (6.5,1) {};
        \node[vertex1_style, ] (a4) at  (7.5,1) {};
                \node[vertex1_style, ] (b1) at  (4,0) {};
                           \node[vertex1_style, ] (b2) at  (5,0) {};
                                   \node[vertex1_style, ] (b3) at  (6.5,0) {};
                                   \node[vertex1_style] (b4) at (7.5,0)  {};
                                        \node[vertex1_style] (c0) at (3.5,.5)  {};
                                         \node[vertex1_style] (c1) at (2.5,.5)  {};
                                        \draw (a1)--(a2);
                                             \draw (a3)--(a4);
                                               \draw (b1)--(b2);
                                             \draw (b3)--(b4);
                                               \draw (c0)--(c1);
                                               \draw (c0)--(a1);
                                                \draw (c0)--(b1);
                                                 \draw[thick, dotted] (a2)--(a3);
                                                   \draw[thick, dotted] (b2)--(b3);
\end{tikzpicture} }
\end{center}
  \caption{ \label{fig0}  \small The graph $P_2(P_m,P_n)$}
  \end{figure}
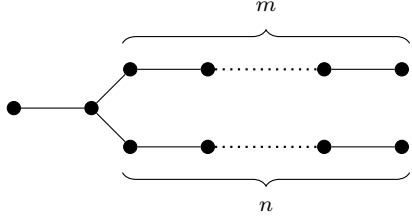

In Section 4, we provide a proof of  Theorems \ref{Aa-main1} and \ref{Aa-main2} below. Such theorems are essentially equivalent, and  generalize Theorem~\ref{hoffman}, which is deducibile from them by setting $\alpha =0$ in their statements. From Theorems \ref{Aa-main1} and \ref{Aa-main2}, we also obtain Theorem~\ref{hoff-new} which can be regarded as a new version of the original theorem by Hoffman.

\begin{thm}[Generalized version-I of Hoffman's theorem]\label{Aa-main1}
For every $\alpha \in [0,1)$ and any non-negative integer $n$,
let
\begin{equation}\label{eq1}
\eta_n(\alpha)=2\alpha+(1-\alpha)\gamma_n(\alpha)^{\frac{1}{2}}+(1-\alpha)\gamma_n(\alpha)^{-\frac{1}{2}},
\end{equation} where $\gamma_0(\alpha)=1$,
$\gamma_1(\alpha) \in (0,1)$ is the only positive root of
\[  \Phi_{1,\alpha}(x) = (1-\alpha)^2 x^2 +2\alpha(1-\alpha) x^{\frac{3}{2}}+\alpha^2x-(1-\alpha)^2, \]
and, for $n \geqslant 2$, $\gamma_n(\alpha) \in (0,1)$ is the  only positive root of
{\small
\begin{equation}\label{PHI} \Phi_{n,\alpha}(x)=(1-\alpha)^2x^{n+1}+2\alpha(1-\alpha)\sum_{i=0}^{n-1}x^{n-i+\frac{1}{2}}+(1-2\alpha+2\alpha^2)\sum_{i=0}^{n-2}x^{i+2}+\alpha^2x-(1-\alpha)^2.
\end{equation}
}
 Then, $$2=\eta_0(0)=\eta_1(0)<\eta_2(0)< \cdots,$$ and
$$ 2=\eta_0(\alpha)<\eta_1(\alpha)<\eta_2(\alpha)<\cdots \qquad (\!\text{ for $\alpha \in (0,1)$})$$
are the all possible  limit points of $A_\alpha$-spectral radius of graphs smaller than $\lim_{n\rightarrow\infty}\eta_n(\alpha)=\Psi(\alpha)$, where $\Psi(\alpha)$ is defined in \eqref{psi}.
\end{thm}

The above sequence $\{ \eta_n(\alpha) \}_{n \geqslant 0}$ of $A_{\alpha}$-limit points can be determined in an alternative way, as Theorem~\ref{Aa-main2} shows.

\begin{thm}[Generalized version-II of Hoffman's theorem]\label{Aa-main2}
For every $\alpha \in [0,1)$ and any non-negative integer $n$,
let
\begin{equation}\label{eq1}
\eta_n(\alpha)=2\alpha+(1-\alpha)\widetilde{\gamma}_n(\alpha)^{\frac{1}{2}}+(1-\alpha)\widetilde{\gamma}_n(\alpha)^{-\frac{1}{2}},
\end{equation} where $\widetilde{\gamma}_0(\alpha)=1$,
$\widetilde{\gamma}_1(\alpha) \in (1,\infty)$ is the only positive root of
\[  \widetilde{\Phi}_{1,\alpha}(x) = (1-\alpha)^2x^2-\alpha^2x-2\alpha(1-\alpha)x^{\frac{1}{2}}-(1-\alpha)^2, \]
and, for $n \geqslant 2$, $\widetilde{\gamma}_n(\alpha) \in (1,\infty)$ is the  only positive root of
{\small
\begin{equation}\label{PHI} \widetilde{\Phi}_{n,\alpha}(x)=(1-\alpha)^2x^{n+1}-\alpha^2x^n-2\alpha(1-\alpha)\sum_{i=1}^{n}x^{n-i+\frac{1}{2}}-(1-2\alpha+2\alpha^2)\sum_{i=1}^{n-1} x^{i}-(1-\alpha)^2.
\end{equation}
}
 Then, $$2=\eta_0(0)=\eta_1(0)<\eta_2(0)< \cdots,$$ and
$$ 2=\eta_0(\alpha)<\eta_1(\alpha)<\eta_2(\alpha)<\cdots \qquad (\!\text{ for $\alpha \in (0,1)$})$$
are the all possible  limit points of $A_\alpha$-spectral radius of graphs smaller than $\lim_{n\rightarrow\infty}\eta_n(\alpha)=\Psi(\alpha)$, where $\Psi(\alpha)$ is defined in \eqref{psi}.
\end{thm}

\begin{thm}[New version of Hoffman's theorem]\label{hoff-new}
For each $n \in \N$, let $\delta_n$ be the only positive root of $$\Phi_{n,0}(x)=x^{n+1}+x^n+x^{n-1}+\cdots+x^2-1.$$
The numbers $\zeta_n=\delta_n^{\frac{1}{2}}+\delta_n^{-\frac{1}{2}}$ satisfy the sequence of inequalities
$$2=\zeta_1<\zeta_2<\cdots,$$ and
are precisely the limit points of $A$-spectral radius of graphs smaller than $$\lim\limits_{n\rightarrow\infty}\zeta_n=\sqrt{2+\sqrt{5}}.$$
\end{thm}

\begin{proof}
Clearly, $\delta_n=\gamma_n(0)$ and $\zeta_n=\eta_n(0)$. From \eqref{puah0}, a direct calculation yields $$\Phi_{n,0}(x)=-x^{n+1}\phi_n \left(\frac{\;1\;}{x}\right).$$
Hence, $\delta_n = \beta_n^{-1}$, and so $\zeta_n = \alpha_n$.
\end{proof}

\section{Preliminaries and technical results}
Apart from Lemma~\ref{alpha-delta}, which is due to Nikiforov, all the other results given in this section without a proof are taken from the reference \cite{JFW-JW-MB} written by the same authors of this paper.

\begin{lem}{\rm \cite{niki}}\label{alpha-delta}
Set $\alpha \in [0,1]$. Let $G$ be a connected graph with maximum degree $\Delta$.
\begin{itemize}
\item[$\mathrm{(i)}$]
Then $\frac{1}{2}\left(\alpha(\Delta+1)+\sqrt{\alpha^2(\Delta+1)^2+4\Delta(1-2\alpha)}\right) \leqslant \rho\!_{_{A_{\alpha}}}\!(G) \leqslant \Delta$.
\item[$\mathrm{(ii)}$]
If $H$ is a proper subgraph of $G$, then $\rho\!_{_{A_{\alpha}}}\!(H)<\rho\!_{_{A_{\alpha}}}\!(G)$.
\item[$\mathrm{(iii)}$] If $0 \leqslant \alpha < \beta \leqslant 1$, then
$ \rho\!_{_{A_{\alpha}}}\! (G) < \rho\!_{_{A_{\beta}}}\!(G)$.
\end{itemize}
\end{lem}

\begin{figure}[h]
\begin{center}
 \resizebox{0.75\textwidth}{!}
{\begin{tikzpicture}[vertex1_style/.style={circle,draw,minimum size=0.17 cm,inner sep=0pt, fill=black},vertex2_style/.style={circle,draw,minimum size=0.2 cm,inner sep=0pt}]
    \foreach \x [count=\p] in {0,...,5} {
        \node[shape=circle,fill=black, scale=0.5] (\p) at (\x*72:.8) {};};
    \draw[thick] (5) arc (-72:72:.8);
     \draw[thick] (3) arc (144:216:.8);
      \draw[thick, dotted] (2) arc (72:144:.8);
       \draw[thick, dotted] (5) arc (-72:-144:.8);
\node[left=3pt] at (.8,0) {\small$v_0$};
\node[above=2pt] at (72:.8) {\small$v_1$};
\node[below=2pt] at (-72:.8) {\small$v_{k-1}$};
\node[left=3pt] at (144:.8) {\small$v_t$};
\node[left=3pt] at (216:.8) {\small$v_{t+1}$};
    \node[shape=circle,fill=black, scale=0.5] (6) at (1.6,.6) {};
      \node[shape=circle,fill=black, scale=0.5] (7) at (1.6,-.6) {};
       \node[shape=circle,fill=black, scale=0.18] (8) at (1.6,.25) {};
         \node[shape=circle,fill=black, scale=0.18] (9) at (1.6,0) {};
                  \node[shape=circle,fill=black, scale=0.18] (10) at (1.6,-.25) {};
      \draw[thick] (1)--(6);
      \draw[thick] (1)--(7);
\node[vertex1_style] (b1) at (3.5,.6)  {};
\node[vertex1_style] (c1) at (3.5,-.6)  {};
 \node[shape=circle,fill=black, scale=0.18] (x8) at (3.5,.25) {};
         \node[shape=circle,fill=black, scale=0.18] (x9) at (3.5,0) {};
                  \node[shape=circle,fill=black, scale=0.18] (x10) at (3.5,-.25) {};
                   \node[shape=circle,fill=black, scale=0.18] (y8) at (10.5,.25) {};
         \node[shape=circle,fill=black, scale=0.18] (y9) at (10.5,0) {};
                  \node[shape=circle,fill=black, scale=0.18] (y10) at (10.5,-.25) {};

 \node[vertex1_style, label=above:\small$\,v_0$] (a1) at  (4.1,0) {};
  \node[vertex1_style, label=above:\small$\,v_1$] (a2) at  (4.9,0) {};
    \node[vertex1_style, label=above:\small$\,v_{t-1}$] (a3) at  (6.2,0) {};
        \node[vertex1_style, label=above:$v_{t}$] (a4) at  (7,0) {};
                \node[vertex1_style, label=above:\small$\,v_{t+1}$] (a5) at  (7.8,0) {};
                           \node[vertex1_style, label=above:\small$\,v_{k-1}$] (a6) at  (9.1,0) {};
                                   \node[vertex1_style, label=above:\small$\,v_{k}$] (a7) at  (9.9,0) {};
                                   \node[vertex1_style] (b2) at (10.5,.6)  {};
                                        \node[vertex1_style] (c2) at (10.5,-.6)  {};
                                        \draw[thick] (a1)--(b1);
                                             \draw[thick] (a1)--(c1);
                                               \draw[thick] (a7)--(b2);
                                             \draw[thick] (a7)--(c2);
                                               \draw[thick] (a1)--(a2);
                                               \draw[thick] (a6)--(a7);
                                                \draw[thick] (a3)--(a5);
                                                 \draw[thick, dotted] (a2)--(a3);
                                                   \draw[thick, dotted] (a5)--(a6);
                                                   \node at (0.2,-1.8) {\small Type I};
                                                     \node at (7,-1.5) {\small Type II};

\end{tikzpicture} }
\end{center}
\vspace{-8mm}
  \caption{ \label{fig1}  \small The two types of internal path.}
  \end{figure}
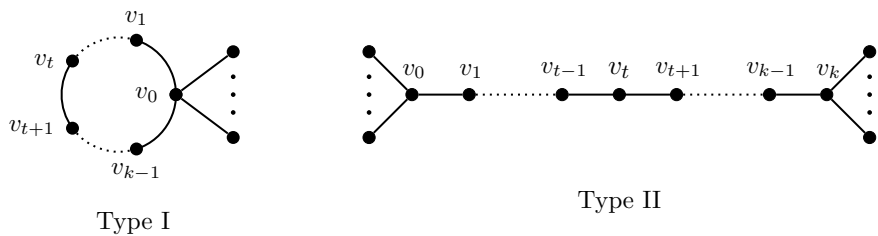

According to  \cite{hof-smi}, an {\it internal path} of a graph $G$ is a walk $v_0 v_1 \dots v_k$ (here $k \geqslant 1$), where the vertices $v_1, \dots, v_k$ are pairwise distinct, $d(v_0) > 2$, $d(v_k) > 2$ and $d(v_i) = 2$ whenever $0 < i < k$. We say that an internal path is of {\em type I} (resp., {\em type II}) if $v_0=v_k$ (resp., $v_0\not=v_k$) (see Fig.~\ref{fig1}).
 In the following lemma and throghout the rest of the paper we denote by $C_n$ the cycle of order $n$ ($n \geqslant 3$), and by $DS_n$  the {\em double snake} of order $n \geqslant 6$, i.e.\ the graph  containing  an internal path of type II $v_0 \dots  v_{n-5}$ such that $d(v_0)=d(v_{n-5})=3$.

\begin{lem}{\rm \cite{JFW-JW-MB}}\label{alpha-internal}
Let $uv$ be an edge of the connected graph $G$, and let $G_{uv}$ be the graph obtained from $G$ by subdividing the edge $uv$ of $G$. Set $\alpha \in [0,1)$.
\begin{itemize}
\item[$\mathrm{(i)}$]
$\rho\!_{_{A_\alpha}}\!(C_n)=2$ and $\rho\!_{_{A_0}}\!(DS_n)=2$;
\item[$\mathrm{(ii)}$]
If $(G,\alpha) \neq (C_n,\alpha)$ and $uv$ is not in an internal path of $G$, then $\rho\!_{_{A_{\alpha}}}\!(G_{uv})>\rho\!_{_{A_{\alpha}}}\!(G)$;
\item[$\mathrm{(iii)}$]
If $(G,\alpha) \neq (DS_n,0)$ and $uv$ belongs to an internal path of $G$, then $\rho\!_{_{A_{\alpha}}}\!(G_{uv})<\rho\!_{_{A_{\alpha}}}\!(G)$.
\end{itemize}
\end{lem}

For each positive integer $n$, we consider the matrix $B_{n}$ obtained from $A_{\alpha}(P_{n+1})$ by deleting the row and column corresponding to a vertex of degree one of the path $P_{n+1}$. We also make use of the following notations:
\begin{equation}\label{kazz1}
\Delta_{\lambda,\alpha} =\sqrt{(\lambda-4\alpha+2)(\lambda-2)} \qquad \text{and}  \qquad h(\lambda)_{\alpha}  = \frac{\lambda -\Delta_{\lambda,\alpha}}{2\alpha(\lambda-2)+2}.
\end{equation}

\begin{lem}{\rm \cite{JFW-JW-MB}}\label{gongshi} Let $n$ be any non-negative integer. After setting
\begin{equation*}
s =\frac{\lambda-2\alpha+\Delta_{\lambda,\alpha}}{2} \;\; and \;\; t= \frac{\lambda-2\alpha-\Delta_{\lambda,\alpha}}{2},
\end{equation*}
\begin{itemize}
\item[$\mathrm{(i)}$]
$\phi(P_{n+1})=\Delta_{\lambda,\alpha}^{-1}((s+\alpha)^2s^{n}-(t+\alpha)^2t^{n})$ for $\alpha \in [0,1)$;
\item[$\mathrm{(ii)}$]
$\phi(B_{n+1})=$\small$\displaystyle \frac{1}{\Delta_{\lambda,\alpha}} \cdot \frac{\alpha}{\left( \alpha(\lambda-2)+1\right)} \left((s+\alpha)^2
\left(s+\frac{(1-\alpha)^2}{\alpha}\right)s^{n}-(t+\alpha)^2\left(t+\frac{(1-\alpha)^2}{\alpha}\right)t^{n}\right).$
\end{itemize}
where the equality (ii) holds for $\alpha \in (0,1)$.
\end{lem}

\medskip
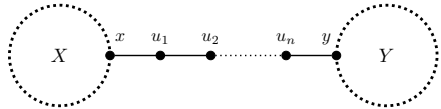
\begin{figure}[h!]
\begin{center}
 \resizebox{0.37\textwidth}{!}
{\begin{tikzpicture}[vertex1_style/.style={circle,draw,minimum size=0.17 cm,inner sep=0pt, fill=black},vertex2_style/.style={circle,draw,minimum size=0.2 cm,inner sep=0pt}]

\draw[dotted, ultra thick] (-.5,0) circle (1);
    \node at (-.5,0) {$X$};
    \node[vertex1_style] (a0) at  (.5,0) {};
     \node[vertex1_style, label=above:\small$u_1$] (a1) at  (1.5,0) {};
      \node[vertex1_style, label=above:\small$u_2$] (a2) at  (2.5,0) {};
       \node[vertex1_style, label=above:\small$u_n$] (a3) at  (4,0) {};
         \node[vertex1_style] (a4) at  (5,0) {};
         \draw[dotted, ultra thick] (5,0) arc (-180:180:1);
         \draw[thick] (a0) -- (a2);
           \draw[thick, dotted] (a2) -- (a3);
                 \draw[thick] (a3) -- (a4);
             \node at (6,0) {$Y$};
              \node at (.7,.36) {\small$x$};
                \node at (4.8,.33) {\small$y$};

\end{tikzpicture} }
\end{center}
\vspace{-6mm}
  \caption{ \label{fig3}  \small The graph $XY(x,y;n)$}
  \end{figure}

\begin{lem}{\rm \cite{JFW-JW-MB}}\label{alpha-gxy} Let $X$ and $Y$
be two vertex-disjoint connected graphs, and let $G_n=XY(x,y;n)$ be the graph  obtained by joining $x \in V(X)$ and $y \in V(Y)$  by a path of length $n+1$ (see Fig. \ref{fig3}). Then,
\begin{equation}\label{mink}
\lim\limits_{n\rightarrow\infty}\rho\!_{_{A_{\alpha}}}\!(G_n)=\max \left\{ \lim\limits_{n\rightarrow\infty}\rho\!_{_{A_{\alpha}}}\!(X_x(P_n)),\lim\limits_{n\rightarrow\infty}\rho\!_{_{A_{\alpha}}}\!(Y_y(P_n)) \right\}.
\end{equation}
\end{lem}

Let $2P_n$ be the disjoint union of two copies of $P_n$ and let $u_1,u_2 \in V(2P_n)$ be end-vertices belonging to different components. For every non-trivial connected graph $G$ and every $u \in V(G)$, let the graph $G_u(P_n,P_n)$ be obtained by adding to $G \cup 2P_n$ the edges $uu_1$ and $uu_2$.

\begin{lem}\label{alpha-gupnpn}{\rm \cite{JFW-JW-MB}}
For $A_{\alpha}$-index of the graph sequence $\{G_u(P_n,P_n)\}_{n \in \N}$,

it has a limit point $\chi'_u(G) \geqslant 2$. If $\chi'_u(G)>2$, then $\chi'_u(G)$ is the largest root of the equation $\Theta(\lambda)_{G,u,\alpha, \infty}=0$, where
\begin{equation*}
\Theta(\lambda)_{G,u,\alpha, \infty}= \left(1-\alpha h(\lambda)_{\alpha}\right)\left(\phi(G)(1-\alpha h(\lambda)_{\alpha})-2\alpha\phi(G)_u+2(2\alpha-1)\phi(G)_uh(\lambda)_{\alpha}\right),\end{equation*}
and $h(\lambda)_{\alpha}$ is defined in \eqref{kazz1}.
\end{lem}
\begin{prop}\label{alpha-p2pnpn}
Let $P_2(P_n,P_n)$ be the graph depicted in Fig.~\ref{fig0}. Then,
$$ \rho\!_{_{A_{\alpha}}}\!(P_2(P_n,P_n))< \Psi(\alpha) = \lim\limits_{n\rightarrow\infty}\rho\!_{_{A{\alpha}}}\!(P_2(P_n,P_n)).$$
Moreover, $\Psi(\alpha)$ is the largest root of
\begin{multline}\label{kazz3} (1-\alpha h(\lambda)_{\alpha}) \left( (1-\alpha h(\lambda)_{\alpha})\lambda^2 \right.\\
\left. +2((\alpha^2+2\alpha-1)h(\lambda)_{\alpha} -2\alpha)\lambda -(6\alpha^2-3\alpha)h(\lambda)_{\alpha}+2\alpha^2+2\alpha-1 \right) .
\end{multline}
\end{prop}
\begin{proof}
The graph $P_2(P_n,P_n)$ is of type $G_u(P_n,P_n)$, where $G=P_2$ and $u$ is any of its vertices.
Hence, $\Psi (\alpha)$ is well-defined by Lemma~\ref{alpha-gupnpn}.
Moreover, by Lemma \ref{alpha-internal}(ii) (or by Lemma~\ref{alpha-delta}(ii) as well), it turns out  that $\rho\!_{_{A_{\alpha}}}\!(P_2(P_n,P_n))$ grows up with $n$; hence, $\rho\!_{_{A_{\alpha}}}\!(P_2(P_n,P_n))<\Psi(\alpha)$.
From Lemma \ref{alpha-delta} and a direct calculation we get
$$\rho\!_{_{A_{\alpha}}}\! (P_2(P_n,P_n)) \geqslant \rho\!_{_{A_0}}\! (P_2(P_n,P_n)) \geqslant \rho\!_{_{A_0}}\! (P_2(P_4,P_4)) >2.$$
From (i) it follows   that $\lim\limits_{n\rightarrow\infty}\rho\!_{_{A_{\alpha}}}\!(P_2(P_n,P_n))$ is the largest root of the following equation:
\begin{equation}\label{zac} \left(1-\alpha h(\lambda)_{\alpha}\right)\left(\phi(P_2)
(1-\alpha h(\lambda)_{\alpha})
-2\alpha\phi(B_1)+2(2\alpha-1)\phi(B_1) h(\lambda)_{\alpha} \right)=0. \end{equation}
We now get  \eqref{kazz3} by plugging $\phi(P_2)=(\lambda-\alpha)^2-(1-\alpha)^2$ and $\phi(B_1)=\lambda-\alpha$ into \eqref{zac}.
\end{proof}

\noindent The software {\em Mathematica\textsuperscript{\textregistered}} provides the next closed-form expression for the function $\Psi(\alpha)$.
\begin{equation}\label{PSI}
\Psi(\alpha)=\frac{\;3\;}{2}\alpha+\frac{\;1\;}{\sqrt{6}}\left(g_0(\alpha)
+\frac{g_1(\alpha)}{g_4(\alpha)}+\frac{g_2(\alpha)}{\sqrt{g_5(\alpha)}} -g_4(\alpha)\right)^{\frac{1}{2}}+\sqrt{\frac{g_5(\alpha)}{12}},
\end{equation}
where $g_0(\alpha)=11\alpha^2-16\alpha+8$,\\[1.5mm]
{\phantom{where }}$g_1(\alpha)=(\alpha-1)^2(2\alpha^2+2\alpha-1)$,\\[1.5mm]
{\phantom{where }}$g_2(\alpha)=\sqrt{27}\alpha(7\alpha^2-12\alpha+6)$,\\[1.5mm]
{\phantom{where }}$g_3(\alpha)=11\alpha^6-86\alpha^5+275\alpha^4-432\alpha^3+358\alpha^2-150\alpha+25$.\\[1.5mm]
{\phantom{where }}$g_4(\alpha)=(1-\alpha)\left((\alpha-1)(17\alpha^2-52\alpha+26)) -\sqrt{27g_3(\alpha)} \right)^{\frac{1}{3}}$,\\[1.5mm]
{\phantom{where }}$\displaystyle g_5(\alpha)=g_0(\alpha)-2\left(\frac{g_1(\alpha)}{g_4(\alpha)}-g_4(\alpha)\right)$.\\

We remind the reader that the values of $g_3(\alpha)$ are not real for $\alpha \in [0,1)$. In fact, like other software packages, {\em Mathematica\textsuperscript{\textregistered}} always chooses the principal branch of fractional powers. This means that,  for every positive real number $a$,
$(-a)^{\frac{1}{3}}$ has to be read as the complex number $\sqrt[3]{a}{\rm e}^{i \frac{\pi}{3}}$. This implies, in particular, that
$$ g_4(0)= \frac{3}{2}+1 + i \left( \sqrt{3} + \frac{3}{2} \right), \qquad g_5(0)=12+6i,$$
and
$$ \begin{array}{ll} \Psi (0) &= \displaystyle \frac{\;1\;}{\sqrt{6}}\left(8
-\frac{1}{g_4(0)}-g_4(0)\right)^{\frac{1}{2}}+\sqrt{\frac{g_5(0)}{12}}  \\[2em]
& \displaystyle =\left( \frac{\sqrt{2+\sqrt{5}}}{2} - i   \frac{\sqrt{\sqrt{5}-2}}{2} \right) + \left( \frac{\sqrt{2+\sqrt{5}}}{2} + i   \frac{\sqrt{\sqrt{5}-2}}{2} \right) =  \sqrt{2+\sqrt{5}}.
\end{array}$$ as expected. The values assumed by $h_4$ in Lemma \ref{K13P5Pn} are computed by the same rule.

Let $G_u(P_n)$ be the graph obtained from two-vertex-disjoint graphs $G$ and $P_n$ by adding a new edge joining a vertex $u$ of $G$ with an end vertex of $P_n$.

\begin{lem}{\rm \cite{JFW-JW-MB}} The $A_{\alpha}$-spectral radius of the graph sequence $\{G_u(P_n)\}_{n \in \N}$ has a limit point $\chi_u(G) \geqslant 2$. Moreover,
\begin{itemize}
\item[$\mathrm{(i)}$]\label{exist}
 if $\chi_u(G)>2$, then $\chi_u(G)$ is the largest root of the equation $$\left(1-\alpha\cdot h(\lambda)_{\alpha} \right)\phi(G)-
\left(\alpha-(2\alpha-1)\cdot h(\lambda)_{\alpha}\right)\phi(G)_u=0,
$$
where $h(\lambda)_{\alpha}$ is defined in \eqref{kazz1};
\item[$\mathrm{(ii)}$]\label{K13P5Pn} if $G=K_{1,3}$ and $u$ is its vertex of degree $3$, then,
$$\lim\limits_{n\rightarrow\infty}\rho\!_{_{A_{\alpha}}}\!((K_{1,3})_u(P_n))=\frac{1}{2}\left(5\alpha+3\sqrt{2-4\alpha+3\alpha^2}\right).$$
\end{itemize}
\end{lem}
\begin{prop}\label{3.8}
Let $u$ be the middle vertex of degree 2 of $G=P_5$. Then,
$$\lim\limits_{n\rightarrow\infty}\rho\!_{_{A_{\alpha}}}\!((P_5)_u(P_n))=2\alpha+\frac{1}{2}(h_1+h_5)^{\frac{1}{2}}+
\frac{1}{2}\left(h_7-h_5+\frac{h_6}{4(h_1+h_5)^{\frac{1}{2}}}\right)^{\frac{1}{2}},$$
where $h_1=4-8\alpha-3\alpha^2,$  $h_2=19\alpha^2+8\alpha-4,$  $h_3=13\alpha^4-32\alpha^3+32\alpha^2-16\alpha+4,$ \\ $h_4=(-416+2496\alpha-6300\alpha^2+8560\alpha^3-6624\alpha^4+2784\alpha^5-502\alpha^6-(172800-2073600\alpha+11453184\alpha^2-
38499840\alpha^3+87733584\alpha^4-142826112\alpha^5+170398080\alpha^6-150197760\alpha^7+97143840\alpha^8-44993664\alpha^9+
14176512\alpha^{10}-2730240\alpha^{11}+243216\alpha^{12})^{\frac{1}{2}})^{\frac{1}{3}},$\\
$h_5=\frac{1}{3} \left( h_2+2^{\frac{1}{3}}\frac{h_3}{h_4}+2^{-\frac{1}{3}}h_4 \right),$ $h_6=512\alpha^3-32\alpha h_2+112(-\alpha+2\alpha^2+\alpha^3),$ $h_7=13\alpha^2-8\alpha+4$.

In particular $\rho\!_{_{A_0}}\!((P_5)_u(P_n))=\Psi(0)= \sqrt{2+\sqrt{5}}$.

\end{prop}

\begin{proof}
The inequalities
$$\rho\!_{_{A_{\alpha}}}\! ((P_5)_u(P_n)) \geqslant \rho\!_{_{A_0}}\! ((P_5)_u(P_n)) \geqslant \rho\!_{_{A_0}}\! ((P_5)_u(P_2)) >2.$$
hold for every $n \geqslant 3$
by Lemma~\ref{alpha-delta} and a direct calculation. Since
$$\rho\!_{_{A_{\alpha}}}\! (K_{1,3}(P_n)) \geqslant \rho\!_{_{A_0}}\! (K_{1,3}(P_n)) \geqslant \rho\!_{_{A_0}}\! (K_{1,3}(P_2)) >2,$$
it follows from (i) that $\lim\limits_{n\rightarrow\infty}\rho\!_{_{A_{\alpha}}}\!((P_5)_u(P_n))$ is the largest root of the  equation
\begin{equation}\label{xxx}
(1-\alpha h(\lambda)_{\alpha})\phi(P_5)
-(\alpha-(2\alpha-1)h(\lambda)_{\alpha})\phi(P_5)_u=0.
\end{equation}
The result comes with the help of {\em Mathematica\textsuperscript{\textregistered}}, once we substitute
$$\phi(P_5)=(\lambda^2-3\alpha\lambda+\alpha^2+2\alpha-1)(\lambda^3-5\alpha \lambda^2+(5\alpha^2+6\alpha -3)\lambda-8\alpha^2+4\alpha),$$
and $\phi(P_5)_u=(\lambda^2-3\alpha\lambda+\alpha^2+2\alpha-1)^2$.

For $\alpha=0$, \eqref{xxx} becomes $\phi(P_5) -\left( \phi(P_2)\right)^2h(\lambda)_0$
whose largest root is $\sqrt{2+\sqrt{5}}$ as already stated in \cite[p. 171]{hoffman}.
\end{proof}

\section{Proofs of Theorems \ref{Aa-main1} and \ref{Aa-main2}}
The next four Lemmas relates the $A_{\alpha}$-spectral radius and $A_{\alpha}$-limit points with  some structural conditions on the graph $G$.
\begin{lem}\label{no-TC}
If $G$ is a connected graph that is neither a tree nor a cycle, then $\rho\!_{_{A_{\alpha}}}\!(G) \geqslant \Psi(\alpha)$.
\end{lem}

\begin{proof}
For each $ n \geqslant 4$, let $L_n$ be the graph obtained from the cycle $C_{n-1}$ by adding a pendant edge. It is easily seen that $L_n$ contains $P_2(P_{\lfloor \frac{n-2}{2} \rfloor},P_{\lfloor\frac{n-2}{2}\rfloor})$ as subgraph.

In our hypotheses, $G$ contains a subgraph isomorphic to $L_m$ for a suitable integer $m \geqslant 4$.
Thus, $\rho\!_{_{A_{\alpha}}}\!(G) \geqslant \rho\!_{_{A_{\alpha}}}\!(L_m) >
\rho\!_{_{A_{\alpha}}}\!(P_2(P_{\lfloor\frac{m-2}{2} \rfloor},P_{\lfloor\frac{m-2}{2}\rfloor}))$, by Lemma \ref{alpha-delta}(ii).
Lemma \ref{alpha-internal} yields $\rho\!_{_{A_{\alpha}}}\!(L_n)>\rho\!_{_{A_{\alpha}}}\!(L_{n+1})$. Hence,
$$\rho\!_{_{A_{\alpha}}}\!(G) \geqslant \rho\!_{_{A_{\alpha}}}\!(L_m)> \lim\limits_{n\rightarrow\infty}\rho\!_{_{A_{\alpha}}}\!(L_n)\geqslant\lim\limits_{n\rightarrow\infty}\rho\!_{_{A_{\alpha}}}\!(P_2(P_{\lfloor\frac{n-2}{2}\rfloor},P_{\lfloor\frac{n-2}{2}\rfloor}))=\Psi(\alpha),$$
where the last equality comes from  Proposition~\ref{alpha-p2pnpn}.
\end{proof}

Let $\mathcal S$ be any infinte set. By `almost all elements of $\mathcal S$' we mean `all elements of $\mathcal S$ apart from a finite number of them'.
\begin{lem}\label{gnam} Let $\mathcal G= \{ G_a\}_{a\in \N} $ be a sequence of graphs such that $\lim_{a\rightarrow\infty}\rho\!_{_{A_{\alpha}}}\!(G_a) < \Psi(\alpha)$. Then $\Delta(G_a) \leqslant 4$ for almost all $a \in \N$.
\end{lem}
\begin{proof} The statement directly follows from Lemma~\ref{alpha-delta}(i). In fact, for each $G_a$ such that $\Delta(G_a) \geqslant 5$,  we have
\[\sqrt{2+\sqrt{5}} = \Psi(0) < \sqrt{5} \leq  \sqrt{\Delta(G_a)} \leq \rho\!_{_{A_0}}\!(G_a). \qedhere
  \popQED
\]
\end{proof}

As usual, we denote by $\diam (G)$ the diameter of $G$, by $N_G(u)$ the neighbourhood of $u$ in $G$, i.e.\ the set of vertices in $V(G)$ adjacent to $u$, and by $d(u,v)$ the number of edges in a shortest path connecting $u$ and $v$.

\begin{lem}\label{gnam2}  Let $\mathcal G= \{ G_a\}_{a\in \N} $ be a sequence of graphs such that $\lim\limits_{a\rightarrow\infty}\rho\!_{_{A_{\alpha}}}\!(G_a) < \Psi(\alpha)$. The set of diameters $\{ \diam (G_a) \, | \, a \in \N \}$ is not bounded above.
\end{lem}
\begin{proof}
It is well-known that
$ \lvert V(G_a) \rvert \leqslant  \Delta(G_a)^{\diam(G_a)}+1$ (see, for instance, \cite{hoffman}).

By Lemma~\ref{gnam}, we have
$$ \lvert V(G_a) \rvert \leqslant  4^{\diam(G_a)}+1$$
for almost all $n \in \N$. Since the graphs in the sequence  $\{ G_a\}_{a\in \N} $ are pairwise distinct, the set $\{ \lvert V(G_a) \rvert \; | a \in \N \}$ cannot be bounded above.
\end{proof}

\begin{lem}~\label{gnam3} Let $T$ be a tree with at least three vertices of degree $3$. Then,
 $\rho\!_{_{A_{\alpha}}}\!(T) \geqslant \Psi(\alpha)$.
 \end{lem}
 \begin{proof} Let $u,v$ and $w$ the vertices in $T$ of degree $3$. Without loss of generality, we can assume that $v$ is the vertex which lies on the path between $u$ and $w$. Consider the minimal subtree $T'$ whose vertex-set contains $\{ u,v,w\} \cup N_T(u) \cup N_T(v) \cup N_T(w)$. Then, it is easy that $P_2(P_m,P_m)$ is a subgraph of $T'$,
  where $m = \min \{ d(v,u), d(v,w) \}$ and $v \in V(P_2(P_m,P_m))$. Since  the path connecting $u$ and $v$ and the one connecting $v$ and $w$ are internal for $T'$, from Lemmas \ref{alpha-delta}(ii), \ref{alpha-internal} and Proposition~\ref{alpha-p2pnpn}, it follows that
\[ \rho\!_{_{A_{\alpha}}}\!(T) \geqslant \rho\!_{_{A_{\alpha}}}\!(T') \geqslant \lim\limits_{k\rightarrow\infty}\rho\!_{_{A_{\alpha}}}\!(P_2(P_k,P_k)) = \Psi(\alpha), \qedhere
  \popQED \]
 \end{proof}

We now have all tools needed to prove the main theorems in this paper.
Let $\mathcal G= \{ G_a\}_{a\in \N} $ be a sequence of graphs such that $\lim_{a\rightarrow\infty}\rho\!_{_{A_{\alpha}}}\!(G_a) < \Psi(\alpha)$.
From Lemma \ref{no-TC} it follows that,  from a suitable integers $\bar{a}$ on, the  $G_a$'s are all trees or  cycles.

Let $\Delta = \max \{ \Delta(G_a) \, | \ a \geq \bar{a} \}$.
If $\Delta=2$, then each $G_a$ for $a \geq \bar{a}$ is either a path or a cycle; this implies that $\lim_{a\rightarrow\infty}\rho\!_{_{A_{\alpha}}}\!(G_a)=2$. The same is true if $\Delta(G_a)=2$ for almost all $a\geqslant \bar{a}$.

Thus, to consider the remaining cases,  it is not restrictive to assume that the number of cycles in $\mathcal G$ is finite, and there exists in $\mathcal G$ a sequence $ \mathcal T= \{T_i\}_{i \in \N}$ of pairwise distinct trees such that $\Delta(T_i) \geqslant 3$ and $\rho\!_{_{A_{\alpha}}}\!(T_i)\rightarrow \sigma$, where $2<\sigma< \Psi(\alpha)$.
 By Lemma~\ref{gnam}, we can also assume that $\Delta(T) \leqslant 4$ for all $ T \in \mathcal T$.

If almost all trees in $\mathcal T$ have a vertex of degree $4$,  then, as a consequence of Lemma \ref{gnam2}, we find in $\mathcal T$ a subsequence $\mathcal T'=\{T_{i_n}\}_{n \in \N}$ such that   $T_{i_n}$ contains $K_{1,3}(P_n)$.

Setting $\omega_1(\alpha)= \lim\limits_{n\rightarrow\infty}\rho\!_{_{A_{\alpha}}}\!(K_{1,3}(P_n))$, from Lemma \ref{K13P5Pn}(ii) we immediately get
$$ \Psi(0) < \frac{3}{\sqrt{2}} < \omega_1(0) \leqslant \lim\limits_{n\rightarrow\infty}\rho\!_{_{A_0}}\!(T_{i_n}) = \lim\limits_{a\rightarrow\infty}\rho\!_{_{A_0}}\!(G_a) < \Psi(0),$$
which is clearly a contradiction. This implies that  the maximum degree of almost all trees in $\mathcal T$ is  necessarily $3$.

\begin{re}
Even if it is not really relevant for the proof of Theorem~\ref{Aa-main1}, it is worth noting that  if $\{ G_a\}_{a \in \N}$ has a limit point and, for almost all graphs,  $\Delta(G_a) \geqslant 4$, then
$$ \lim\limits_{a\rightarrow\infty} \rho\!_{_{A_{\alpha}}}\!(G_a)
\geqslant \omega_1(\alpha)> \Psi(\alpha) \quad \text{for all $\alpha \in [0,1)$}. $$
In fact, employing {\em Mathematica\textsuperscript{\textregistered}}, after running the command `{\rm NMaxValue[$\Psi(\alpha)-\omega_1(\alpha)$, $\alpha$]}', we get  $$\max\{\Psi(\alpha)-\omega_1(\alpha)|\alpha \in \mathbb{R}\} = -0.0716+.$$
\end{re}

Since the $A_{\alpha}$-spectral radius of almost all graphs in $\mathcal G$ is less than $\Psi(\alpha)$, by Lemma~\ref{gnam3} it follows that only a finite number of them has more than two vertices of degree $3$.

We now show that if almost all trees in $\mathcal T$ have two vertices of degree $3$, then the distance between them is unbounded.

Assuming the contrary, there exists a suitable $m \in \N$ such that $\mathcal T$ contains a subsequence $\mathcal T'' =
\{ T_{j_n} \}_{n \geqslant m+3}$ with the following property: $T_{j_n}$ contains the tree $T_{m,n}$ obtained from a path $u_1u_2\dots v_{m+2}\dots v_{n}$ and two isolated vertices $v,w$ by joining $v_2$ to $v$ and $v_{m+2}$ to $w$.
By construction, $u_2\dots u_{m+2}$ is an internal path for $T_{m,n}$, therefore Lemma~\ref{alpha-delta}(ii) implies that $ \rho\!_{_{A_{\alpha}}}\!(T_{j_n}) \geqslant \rho\!_{_{A_{\alpha}}}\!(T_{m,n}) > \rho\!_{_{A_{\alpha}}}\!(P_2(P_{n-m},P_{n-m}))$. By taking the limits, and recalling Proposition~\ref{alpha-p2pnpn}, we get $\lim\limits_{a\rightarrow\infty}\rho\!_{_{A_{\alpha}}}\!(G_a) \geqslant \lim\limits_{n\rightarrow\infty}\rho\!_{_{A_{\alpha}}}\!(T_{j_n}) \geqslant \lim\limits_{n\rightarrow\infty}\rho\!_{_{A_{\alpha}}}\!(P_2(P_n,P_n)) = \Psi(\alpha),$
against our assumption.\smallskip

So far, we have proved that only a finite number of trees in $\mathcal T$ has more than two vertices of degree $3$, and if almost all trees in $\mathcal T$ have two vertices of degree $3$, then the distance between them is unbounded. This means that,  by Lemma \ref{alpha-gxy}, the sequence $\mathcal T$  can be possibly replaced with another  sequence of trees $\mathcal T'''= \{ T'''_i \}_{i \in \N }$ such that every $T'''_i$ has only one vertex with degree $\Delta(T''''_i) =3$, and $\lim_{i\rightarrow\infty}\rho\!_{_{A_{\alpha}}}\!(T'''_i)= \lim_{i\rightarrow\infty}\rho\!_{_{A_{\alpha}}}\!(T_i)<\Psi(\alpha)$.

We now show that, for almost all trees in   $\mathcal T'''$, the unique vertex of degree $3$ is adjacent to a pendant vertex. Otherwise, we would find a subsequence
$\{T'''_{j_h}\}_{h \in \N}$ of $\mathcal T'''$ such that each $T'''_{j_h}$ contains $(P_5)_u(P_h)$ defined in and by Proposition~\ref{3.8},
$$ \Psi(0) > \lim\limits_{i\rightarrow\infty}\rho\!_{_{A_0}}\!(T_i) \geqslant  \lim\limits_{h\rightarrow\infty}\rho\!_{_{A_0}}\!(T'''_{j_h}) \geqslant \lim\limits_{h\rightarrow\infty}\rho\!_{_{A_0}}\!((P_5)_u(P_h)) = \Psi(0),$$
a contradiction.

\begin{re} Let $\omega_2(\alpha)=\lim_{n\rightarrow\infty}\rho\!_{_{A_{\alpha}}}\!((P_5)_u(P_n))$ (recall that $u$ is middle vertex of degree 2 in $P_5$). The presence of an arbitrarily big $(P_5)_u(P_h)$ inside almost all $T_i$'s implies that
$\lim_{i\rightarrow\infty}\rho\!_{_{A_{\alpha}}}\!(T_i) \geqslant \omega_2(\alpha) > \Psi(\alpha)$ for each $\alpha \in [0,1)$. In fact, using the command `{\rm FindMinVal\-ue[$g_2(\alpha)-\Psi(\alpha),\{\alpha,0\}$]}', in {\rm Mathematica}\textsuperscript{\textregistered}, we get
$$\max\{\, \omega_2(\alpha)-\Psi(\alpha)|\alpha \in [0,\infty) \, \} = 2.22045\times 10^{-16}>0.$$
\end{re}

We have seen that almost each $T'''_{j_h}$ is a a tree of type $P_2(P_{n_h},P_{m_h})$. In other words, if this is the case, the removal of the unique vertex of degree $3$ gives rise to the disjoint union of the three paths $P_1, P_{n_h}$ and $P_{m_h}$; yet, the two `rays' $P_{n_h}$ and $P_{m_h}$ cannot be both arbitrarily long, otherwise $\Psi(\alpha) > \lim\limits_{i\rightarrow\infty}\rho\!_{_{A_{\alpha}}}\!(T'''_{i}) \geqslant \lim\limits_{n\rightarrow\infty}\rho\!_{_{A_{\alpha}}}\!(P_2(P_n,P_n)) = \Psi(\alpha),$ which is clearly impossible.

From the  discussion above, it follows that the set of limit points we seek are the numbers:
$$\lim_{m\rightarrow\infty}\rho\!_{_{A_{\alpha}}}\!(P_2(P_m,P_n))= \eta_n(\alpha)  \qquad \text{(for  $n \in \N$),}$$
where $P_2(P_m,P_n)$ is the graph depicted in Fig.~\ref{fig0}. Note that $\eta_n(\alpha)>2$, unless $n=1$ and $\alpha=0$.
In fact, fixed $ n \geqslant 2$, $P_2(P_2,P_2)$ is a proper subgraph of $P_2(P_m,P_n)$ for all $m>2$; hence,
\[ \eta_n(\alpha) =  \lim_{m\rightarrow\infty}\rho\!_{_{A_{\alpha}}}\!(P_2(P_m,P_n)) > \rho\!_{_{A_{\alpha}}}\!\left(P_2(P_2,P_2)\right)
\geqslant \rho\!_{_{A_0}}\! \left(P_2(P_2,P_2) \right) = \sqrt{\frac{5+\sqrt{13}}{2}} >2. \]

When $n=1$, by \cite[Proposition~3.6]{WWXB} it follows that $\eta_1(0)=2$, and $\eta_1(\alpha) >2$ for $\alpha>0$ for Lemma~\ref{alpha-delta}(ii).

For the rest of the proof, we assume that $(n,\alpha) \not= (1,0)$. Since  $\eta_n(\alpha)>2$, by Lemma \ref{exist}(i), it follows that $\eta_n(\alpha)$ is the largest root of the equation
\begin{equation}\label{kkk}
\left(1-\alpha h(\lambda)_{\alpha}\right)\phi(P_{n+2})-
\left( \alpha- (2\alpha-1) h(\lambda)_{\alpha} \right)\phi(B_1)\phi(B_n)=0.
\end{equation}
We now set
\begin{equation}\label{Deogratias}
\lambda=(1-\alpha)\theta+\frac{1-\alpha}{\theta}+2\alpha.
\end{equation}
Note that $\lambda>2$. So, $\theta > 0$ and $\theta \neq 1$.

An obvious substitution leads to the identity
\begin{equation}\label{hhh1} h(\lambda)_{\alpha} = \frac{\theta}{2(\alpha+\theta(1-\alpha))(1-\alpha (1-\theta))} \cdot \left( (1-\alpha)\theta + \frac{1-\alpha}{\theta} +2\alpha - \left\lvert \frac{(1-\alpha)(1-\theta^2)}{\theta}
\right\rvert \right). \end{equation}

We next distinguish the following two cases.

{\bf Case 1}. $0 <\theta<1$. Then \eqref{hhh1} is equivalent to
\begin{equation}\label{hhh} h(\lambda)_{\alpha} = \frac{2\theta(\alpha+\theta(1-\alpha))}{2(\alpha+\theta(1-\alpha))(1-\alpha (1-\theta))} =
 \frac{\theta}{1-\alpha (1-\theta)}. \end{equation}
Thus, Equation \eqref{kkk} becomes
\begin{equation}\label{qqq} \frac{1-\alpha}{1-\alpha(1-\theta)} \left(\phi( P_{n+2})-(\alpha (1-\theta)+\theta)\phi(B_1)\phi(B_n) \right)=0.
\end{equation}
By Lemma \ref{gongshi}, \eqref{Deogratias} and \eqref{hhh}, we get
\begin{equation}\label{Pn+2} \scalemath{.85}{ \phi(P_{n+2}) = \frac{\theta (1-a)^{n}}{1-\theta^2} \left( (1-\alpha)^2 \left( \frac{1}{\theta^{n+3}}-\theta^{n+3}\right)+2\alpha(1-\alpha) \left(\frac{1}{\theta^{n+2}}-\theta^{n+2}\right)+\alpha^2 \left(\frac{1}{\theta^{n+1}}-\theta^{n+1}\right) \right)}, \end{equation}
\begin{equation}\label{B1}
\phi(B_1) = (1-\alpha)\theta+\frac{1-\alpha}{\theta}+\alpha,\end{equation}
and
\begin{equation}\label{Bn}
\phi(B_n)= \frac{\theta  (1-a)^{n-1} }{1-\theta^2} \left(  \left( \frac{1}{\theta^{n+1}}-\theta^{n+1}\right)(1-a) + \left(\frac{1}{\theta^{n}}-\theta^{n}\right) a \right).
\end{equation}
When we plug the three expressions above in \eqref{qqq}, such equation becomes
\begin{multline}\label{uhm}  \displaystyle \frac{(1-\alpha)^n}{(1-\theta^2)\theta^{n+2}}\left( (1-\alpha)^2\theta^{2n+4}-(2\alpha^2-2\alpha)\theta^{2n+3}+\alpha^2\theta^{2n+2} \right.\\
\left.-(1-\alpha)^2\theta^4+(2\alpha^2-2\alpha)\theta^3-(2\alpha^2-2\alpha+1)\theta^2+(1-\alpha)^2
\right) =0,
\end{multline}

Once we set $\theta^2=x$, then $0<x <1$, and Equation \eqref{uhm} is equivalent to $\Omega_{n,\alpha}(x)=0$, where
\begin{equation*}
\begin{split}
\Omega_{n,\alpha}(x)&=  \frac{(1-\alpha)^n}{(1-x)x^{\frac{n}{2}+1}}\left( (1-\alpha)^2x^{n+2}+2\alpha (1-\alpha)x^{n+\frac{3}{2}}+\alpha^2x^{n+1}-(1-\alpha)^2x^2-2\alpha(1-\alpha)x^{\frac{3}{2}}\right.\\
&\phantom{==}-(2\alpha^2-2\alpha+1)x+(1-\alpha)^2\Big)\\
&=\frac{(1-\alpha)^n}{(1-x)x^{\frac{n}{2}+1}}(x-1)\Phi_{n,\alpha}(x)\\
&= -\frac{(1-\alpha)^n}{x^{\frac{n}{2}+1}}\Phi_{n,\alpha}(x),
\end{split}
\end{equation*}
with $\Phi_{n,\alpha}(x)$ being in the statement of Theorem~\ref{Aa-main1}.  Therefore, we only consider the positive root of $\Phi_{n,\alpha}(x) = 0$. It is easily verified that, for $(n, \alpha) \not=(1,0)$,
$$\Phi_{n,\alpha}(0)=-(1-\alpha)^2<0, \quad
\Phi_{n,\alpha}(1)= 2\alpha(1-\alpha)n +(1-\alpha^2)(n-1)+\alpha^2>0,$$
and
$$ \frac{\mathrm{d} \Phi_{n,\alpha}(x)}{\mathrm{d}x}>0 \qquad \text{for $x>0$.}$$
Hence, $\gamma_n(\alpha)$ is the {\em only} positive root of \eqref{PHI} and satisfies \eqref{eq1} as claimed.

Theorem \ref{Aa-main1} also holds for $(n, \alpha) =(1,0)$.
In fact, $\gamma_1(0)=1$ is the only positive root  of both  $\Omega_{1,0}(x)= (x-1)^2(x+1)$ and $\Phi_{1,0}(x)= x^2-1$.

From  Proposition~\ref{alpha-p2pnpn}  it follows that $\lim_{n\rightarrow\infty}\eta_n(\alpha)=\Psi(\alpha)$.

Since $P_2(P_m,P_n)$ is a subgraph of $P_2(P_m,P_{n+1})$, by
Lemma~\ref{alpha-delta}(ii)  we obtain
\begin{equation}\label{basta}
\eta_1(\alpha) \leqslant \eta_2(\alpha) \leqslant \cdots \leqslant \eta_n(\alpha) \leqslant \eta_{n+1}(\alpha) \leqslant \cdots \qquad \text{ for $\alpha \in [0,1)$.}
\end{equation}

The proof will be over once we prove that all inequalities in \eqref{basta} are strict.

By contradiction, suppose that $ \eta_{\bar{n}}(\bar{\alpha})$ is equal  to $ \eta_{\bar{n}+1}(\bar{\alpha})$  for a suitable pair $(\bar{n},\bar{\alpha})$. Since the function $\lambda=\lambda (\theta)$ in \eqref{Deogratias} is strictly decreasing in the interval $(0,1)$, we also have
$ \gamma_{\bar{n}}(\bar{\alpha}) = \gamma_{\bar{n}+1}(\bar{\alpha})$.

From \eqref{PHI} it comes out that the number \begin{equation*}
f(x,\alpha)=(x^2-2x^{\frac{3}{2}}+2x-1)\alpha^2+2(1-x+x^{\frac{3}{2}}-x^2)\alpha+x^2+x-1
\end{equation*}
computes the difference $\Phi_{n+1, \alpha}(x)-x\Phi_{n, \alpha}(x)$ {\em whatever $n$ we choose in $\N$}.
Under our assumptions,
$$f(\gamma_{\bar{n}}(\bar{\alpha}) ,\alpha)=\Phi_{\bar{n}+1, \alpha}(\gamma_{\bar{n}}(\bar{\alpha}) )-\gamma_{\bar{n}}(\bar{\alpha}) \Phi_{\bar{n}, \alpha}(\gamma_{\bar{n}}(\bar{\alpha}) )=0 -  \gamma_{\bar{n}}(\bar{\alpha}) \cdot 0         =0.$$
Through an obvious iterative argument, we see that the non-zero number $\gamma_{\bar{n}}(\bar{\alpha})$ should be a root of
$\Phi_{n, \bar{\alpha}}(x)$ for all $n \in \N$. But this is not possible, since the difference
$$ \Phi_{2, \bar{\alpha}}(x) - \Phi_{1, \bar{\alpha}}(x) = - x^2  \left( (1-\bar{\alpha})^2 x +2\bar{\alpha}(1-\bar{\alpha})\sqrt{x} -\bar{\alpha}^2 \right)$$
is negative for all $x \in (0,1)$.

Consequently, we have shown Theorem \ref{Aa-main1}.\medskip

{\bf Case 2}. $\theta>1$.  In this case, \eqref{hhh} and \eqref{qqq} should be replaced by
$$ h(\lambda) = \frac{1}{\theta-\alpha(\theta-1)}, \quad \text{and} \quad \frac{1-\alpha}{\theta-\alpha(\theta-1)} \left(\theta\phi( P_{n+2})-(1+\alpha (\theta-1))\phi(B_1)\phi(B_n) \right)=0,
 $$
whereas the expression for $\phi(P_{n+2})$, $\phi(B_1)$ and $\phi(B_n)$ would remain formally identical to \eqref{Pn+2}, \eqref{B1} and \eqref{Bn}. After some calculations, \eqref{qqq} becomes
\begin{multline*}
\frac{(1-\alpha)^n}{(x-1)x^{\frac{n}{2}+1}}\left( (1-\alpha)^2x^{n+2}-(2\alpha^2-2\alpha+1)x^{n+1}+(2\alpha^2-2\alpha)x^{n+\frac{1}{2}} -(1-\alpha)^2x^{n} \right.\\
\left. +\alpha^2x-(2\alpha^2-2\alpha)x^{\frac{1}{2}} +(1-\alpha)^2
\right) =0.
\end{multline*}
Set $\theta^2=x$. Then $x>1$, and the above equation is equivalent to  $\widetilde\Omega_{n,\alpha}(x)=0$, where
\begin{equation*}
\begin{split}
\widetilde{\Omega}_{n,\alpha}(x)&=  \frac{(1-\alpha)^n}{(x-1)x^{\frac{n}{2}+1}}\left( (1-\alpha)^2x^{n+2}-(2\alpha^2-2\alpha+1)x^{n+1}+(2\alpha^2-2\alpha)x^{n+\frac{1}{2}} -(1-\alpha)^2x^{n} \right.\\
&\phantom{==}+\alpha^2x-(2\alpha^2-2\alpha)x^{\frac{1}{2}} +(1-\alpha)^2\Big).\\
&=\frac{(1-\alpha)^n}{x^{\frac{n}{2}+1}}\widetilde{\Phi}_{n,\alpha}(x)
\end{split}
\end{equation*}
with $\widetilde{\Phi}_{n,\alpha}(x)$ being  as in the statement of Theorem~\ref{Aa-main2}. A direct calculation leads to
\begin{equation}\label{equi}
\Omega_{n,\alpha}(x) = x^{n+2} \widetilde{\Omega}_{n,\alpha}\left(\frac{1}{x} \right) \qquad \mbox{and} \qquad \Phi_{n,\alpha}(x) = -x^{n+1} \widetilde{\Phi}_{n,\alpha} \left(\frac{1}{x} \right).
\end{equation}
From Case 1, it follows that $\widetilde{\gamma}_n(\alpha) = \frac{1}{\gamma_n(\alpha)}$ is the only positive root greater than 1 of $\widetilde{\Phi}_{n,\alpha}(x)$. Thereby, \begin{equation*}
\begin{split}
\eta_n(\alpha) &= 2\alpha+(1-\alpha)(\gamma_n(\alpha))^{\frac{1}{2}} + (1-\alpha)(\gamma_n(\alpha))^{-\frac{1}{2}}\\
               &= 2\alpha+(1-\alpha)(\widetilde{\gamma}_n(\alpha))^{\frac{1}{2}} + (1-\alpha)(\widetilde{\gamma}_n(\alpha))^{-\frac{1}{2}}.
\end{split}
\end{equation*}
What is left to prove for Theorem \ref{Aa-main2} also comes from Case 1, and the proofs of Theorem \ref{Aa-main1} and Theorem \ref{Aa-main2}.

\section{(Signless) Laplacian matrix}
 Inspired by Hoffman's theorem, Guo \cite{guo-lim} and Wang et al. \cite{wang-lim} respectively determined the Laplacian and signless Laplacian limit points smaller than $2+\omega+\omega^{-1}$ and $2+\varepsilon$, where $\omega =
\frac{{\;}1{\;}}{3}\left((19 + 3\sqrt{33})^{\frac{{\;}1{\;}}{3}} + (19 - 3\sqrt{33})^{\frac{{\;} 1{\;}}{3}}+1 \right)$ and
$\varepsilon=\frac{{\;}1{\;}}{3}\left((54 - 6\sqrt{33})^{\frac{{\;}1{\;}}{3}} + (54 + 6\sqrt{33})^{\frac{{\;} 1{\;}}{3}} \right)$. Note that $\varepsilon=\omega + \omega^{-1} =2.38+$, and the related proofs of \cite[Theorem~3.5]{guo-lim} and \cite[Theorem~3.1]{wang-lim}  only involve trees. It is well-known that  the Laplacian and signless Laplacian spectra of bipartite graphs are equal \cite{gro-mer-sun}. Thus, we can state the cited results of \cite{guo-lim} and  \cite{wang-lim} in a proposition.

\begin{prop}{\rm \cite{guo-lim,wang-lim}}\label{QL-limit}
Let $\mu_0 = 1$ and $\mu_n $$(n \geq 1)$ be the largest positive root of $$f_n(x) = x^{n+1}-(1+x+\cdots+x^{n-1})(\sqrt{x}+1)^2.$$ Let
$\kappa_n = 2+ \mu_n^{\frac{1}{2}} + \mu_n^{\frac{-1}{2}}$. Then,
$$4 = \kappa_0 < \kappa_1 < \kappa_2 < \cdots$$ are all the limit points
of $L$-spectral radius or $Q$-spectral radius of graphs less than {\small $\lim\limits_{n\rightarrow\infty}\kappa_n=2+\varepsilon$},
where $\varepsilon = \frac{{\;}1{\;}}{3}\left((54 - 6\sqrt{33})^{\frac{{\;}1{\;}}{3}} + (54 + 6\sqrt{33})^{\frac{{\;} 1{\;}}{3}} \right) = 2.38+$.
\end{prop}

The proofs of Theorems \ref{Aa-main1} and \ref{Aa-main2} are also essentially limited to trees. Therefore, on the one side, Proposition \ref{QL-limit} can be seen as corollary of Theorem \ref{Aa-main2} obtained by evaluating $\alpha$ at $ \frac{1}{2}$; on the other side,  we retrieve from  Theorem \ref{Aa-main1} an alternative statement concerning the limit points of (signless) Laplacian spectral radius of graphs.

\begin{thm}\label{ult}
Let $\vartheta_0=1$, $\vartheta_1$ be the only positive root of
\[  \varphi(x) =  x^2 +2 x^{\frac{3}{2}}+x-1, \]
and, for $n \geqslant 2$,
 $\vartheta_n$ be the smallest positive root of
$$\varphi_n(x)=x^{n+1}+2 x^{\frac{3}{2}}+2\sum_{i=0}^{n-2}x^{i+2}+x-1.$$
Let $\xi_n=2+\vartheta_n^{\frac{\;1}{2}}+\vartheta_n^{-\frac{\;1}{2}}$. Then,
$$4=\xi_0<\xi_1<\xi_2<\cdots$$
are all the limit points of (signless) Laplacian spectral radius of graphs smaller than $$\lim\limits_{n\rightarrow\infty}\xi_n=2+\varepsilon, \qquad \text{where} \quad  \varepsilon=\frac{{\;}1{\;}}{3}\left((54 - 6\sqrt{33})^{\frac{{\;}1{\;}}{3}} + (54 + 6\sqrt{33})^{\frac{{\;} 1{\;}}{3}} \right).$$
\end{thm}

\begin{proof}
Recall that for every bipartite graph $G$, $L(G)=Q(G)=2A_{1/2}(G)$.
From Theorem \ref{Aa-main1} and a direct calculation we get $$\varphi_n(x)=4\Phi_n\left(x,\frac{\;1}{2}\right), \quad \vartheta_n = \gamma_n \left( \frac{1}{2} \right) \quad \mbox{and} \quad \xi_n=2\eta_n \left(\frac{\;1}{2} \right).$$
By evaluating \eqref{PSI} at $\alpha=1/2$, the software {\em Mathematica\textsuperscript{\textregistered}} gives  $2\Psi \left( \frac{1}{2} \right) = 2+\varepsilon$.
In order to verify that Theorem~\ref{ult} is consistent with Theorem \ref{QL-limit},
we just check that
$$\varphi_n(x) = x^{n+1}\chi_n\left(\frac{1}{x}\right).$$
Thus, $\vartheta_n=1/\mu_n$ and, consequently, $\xi_n=\kappa_n$.
\end{proof}

\section{Concluding remarks}
About fifty years after its publication, the paper \cite{hoffman} is still inspiring people working in Spectral Graph Theory. The work presented in this paper not only provides a new version but also presents two generalized results of Hoffman's theorem about the limit points of adjacency spectral radius of graphs. In contrast, we take advantage of the software {\em Mathematica\textsuperscript{\textregistered}}, which plays an influential role in Section 3 and along the proofs of Theorems \ref{Aa-main1} and \ref{Aa-main2}. After \cite{WWXB}, the main results in the paper can be seen as a second step towards the more general problem of determining all the limit points of $A_{\alpha}$-spectral radius of graphs.

Very recently, to estimate the maximum cardinality of equiangular lines in the $n$-dimensional Euclidean space, Jiang and Polyanskii \cite{jiang-poly} applied Hoffman's theorem and the related results in \cite{BN,CDG,she} to give a forbidden subgraphs characterization of graphs with bounded adjacency spectral radius. This is a novel application. To get further results in the same vein, it could be important to solve Problem~1 and prove or disprove Conjecture~1 below.\\

\noindent
{\bf Problem 1}.
Characterize all the connected graphs with $A_\alpha$-spectral radius between $2$ and $\Psi(\alpha)$.\\

\noindent
{\bf Conjecture 1}
Let $\alpha \in [0,1)$. For any $\Upsilon(\alpha) \geq \Psi(\alpha) $, there exists a sequence of graphs $\{G_i\}_{i\in \N}$  such that $\lim\limits_{i\rightarrow\infty}\rho_{_{A_\alpha}}(G_i) = \Upsilon(\alpha)$.

\medskip\medskip


\noindent{\bf Acknowledgments.}
The first  author is supported for this research by the National Natural Science Foundation of China (No. 11971274).

\small{

}

\end{document}